\documentclass[11pt]{amsart}
\usepackage[left=2cm,top=2cm,right=3cm,nofoot]{geometry}
\usepackage{amsmath}%
\usepackage{amsfonts}%
\usepackage{amssymb}%
\usepackage{graphicx}

\usepackage{color}

\newtheorem{theorem}{Theorem}
\theoremstyle{plain}

\newtheorem{corollary}{Corollary}

\newtheorem{definition}{Definition}

\newtheorem{proposition}{Proposition}
\newtheorem{remark}{Remark}

\newcommand{\ve}{\varepsilon}
\newcommand{\1}{\mathbf 1}
\definecolor{OliveGreen}{cmyk}{0.64, 0, 0.95, 0.40}

\numberwithin{equation}{section}
\newcommand{\M} {\mathbb M}
\newcommand{\bG} {\mathbb G}
\newcommand{\R} {\mathbb R}
\newcommand{\bM} {\mathbb M}
\begin{document}
\title[Volume Comparison]{Volume and distance comparison theorems for  sub-Riemannian manifolds}
\author{Fabrice Baudoin}
\address{Department of Mathematics\\Purdue University \\
West Lafayette, IN 47907} \email[Fabrice Baudoin]{fbaudoin@math.purdue.edu}
\thanks{First author supported in part by
NSF Grant DMS 0907326}
\author{Michel Bonnefont}
\address{Institut de Math\'ematiques de Bordeaux\\ Universit\'e de Bordeaux 1\\ Talence 33405, France } \email[Michel Bonnefont]{michel.bonnefont@math.u-bordeaux1.fr}
\author{Nicola Garofalo}
\address{Department of Mathematics\\Purdue University \\
West Lafayette, IN 47907} \email[Nicola
Garofalo]{garofalo@math.purdue.edu}
\thanks{Third author supported in part by NSF Grant DMS-0701001 and by NSF Grant DMS-1001317}
\author{Isidro H Munive}
\address{Mathematical Analysis, Modelling and Applications\\SISSA \\
Via Bonomea 265, Trieste, ITALY  34136} \email[Isidro H Munive]{imunive@sissa.it}
\thanks{Fourth author supported in part by the third author's NSF Grants DMS-0701001 and DMS-1001317}
\begin{abstract}
In this paper we study global  distance estimates and uniform local volume estimates in a large class of sub-Riemannian manifolds. Our main device is the generalized curvature dimension inequality introduced by the first and the third author in \cite{BG1} and its use to obtain sharp inequalities for solutions of the sub-Riemannian heat equation. As a consequence, we obtain a Gromov type precompactness theorem for the class of sub-Riemannian manifolds  whose generalized Ricci curvature is bounded from below in the sense of \cite{BG1}.
\end{abstract}
\maketitle

\tableofcontents

\section{Introduction}

The goal of the present work is to study volume and distance comparison estimates on sub-Riemannian manifolds that satisfy the generalized curvature dimension inequality introduced in \cite{BG1}.  We in particular prove a global doubling property in the possibly negative curvature case which complements  the volume estimates obtained in \cite{BBG}, where the curvature was always supposed to be non negative. The distance estimates we obtain, and the methods to prove them are new, but in the non negatively curved Sasakian case that was treated in \cite{BB-sasakian}. As a consequence of the global doubling property , we obtain a Gromov type precompactness theorem for the class of sub-Riemannian manifolds that satisfy the generalized curvature dimension inequality and, as a consequence of the distance comparison theorem, we obtain  Fefferman-Phong type  subelliptic estimates.

\

To put the results we obtain in perspective, let us point out that distance and volume estimates in sub-Riemannian geometry have been extensively studied in the literature. But most of the obtained results are of local nature. More precisely, let $(\M,g)$ be a smooth and connected Riemannian manifold.  Let us assume that there exists on $\M$ a family of vector fields $\{ X_1, \cdots X_d \}$ that satisfy the bracket generating condition. We are interested in the sub-Riemannian structure on $\M$ which is given by the vector fields  $\{ X_1, \cdots X_d \}$. In sub-Riemannian geometry the Riemannian distance $d_R$ of $\mathbb M$ is most of the times confined to the background (see in this
regard the discussion in section 0.1 of Gromov's \emph{Carnot-Carath\'eodory spaces seen from within} in \cite{Be}).
There is another distance on $\mathbb M$, that was introduced by
Carath\'eodory in his seminal paper \cite{Car}, which plays a
central role. A piecewise $C^1$ curve $\gamma:[0,T]\to \mathbb M$ is
called subunitary at $x$ if for every $\xi\in T^*_x\mathbb M$ one
has
\[
g(\gamma'(t),\xi)^2 \le
 \sum_{i=1}^d g(X_i(\gamma(t)),\xi)^2.
\]
We define the subunit length of $\gamma$ as $\ell_s(\gamma) = T$. If
we indicate with $S(x,y)$ the family of subunit curves such that
$\gamma(0) = x$ and $\gamma(T) = y$, then thanks to the fundamental
accessibility theorem of Chow-Rashevsky the connectedness of
$\mathbb M$ implies that $S(x,y) \not= \varnothing$ for every
$x,y\in \mathbb M$, see \cite{Ch}, \cite{Ra}. This allows to define
the sub-Riemannian distance on $\mathbb M$ as follows
\[
d(x,y) = \inf\{\ell_s(\gamma)\mid \gamma\in S(x,y)\}.
\]
We refer the reader to the cited contribution of Gromov to
\cite{Be}, and to the opening article by Bella\"iche in the same
volume. Another elementary consequence of the Chow-Rashevsky theorem is that the identity map $i:(\mathbb M,d) \hookrightarrow (\mathbb M,d_R)$ is continuous and thus, the topologies of $d_R$ and $d$ coincide. 
Several fundamental properties of the metric $d$ have been discussed in the seminal paper by Nagel, Stein and  Wainger \cite{NSW}. 
In particular, the following local distance comparison theorem was proved in \cite{NSW}.

\begin{theorem}[Nagel-Stein-Wainger, \cite{NSW}]\label{T:comparison} For any connected set $\Omega \subset \mathbb M$ which is bounded in the distance $d_R$ there exist $K= K(\Omega)>0$, and $\epsilon =
\epsilon(\Omega)>0$, such that
\[
d(x,y) \leq  C d_R(x,y)^\epsilon,\ \ \ x, y \in \Omega.
\]
\end{theorem}

The following result also proved  in \cite{NSW} provides a uniform local control of the growth of the metric balls in $(\bM,d)$.

\begin{theorem}[Nagel-Stein-Wainger, \cite{NSW}] \label{T:doubling}
For any $x\in \bM$ there exist constants $C(x), R(x)>0$ such that
with $Q(x) = \log_2 C(x)$ one has
\[
\mu(B(x,tr)) \ge C(x)^{-1} t^{Q(x)} \mu(B(x,r)),\ \ \ 0\le t\le 1,\
0<r\le R(x).
\]
Given any compact set $K\subset \bM$ one has
\[
\underset{x\in K}{\inf}\ C(x) >0,\ \ \ \underset{x\in K}{\inf}\ R(x)
>0.
\]
\end{theorem}

These theorems and the methods used  to prove them are local in nature. The goal of the present paper is to obtain global analogues for a large class of sub-Riemannian manifolds. 

\

To fix the ideas, let us present our main results in the context of Sasakian manifolds but we stress that the class of sub-Riemannian structures to which our results apply is much larger than the class of Sasakian manifolds.

Let $\mathbb{M}$ be a complete strictly pseudo convex  CR  manifold with real dimension $2n +1$. Let $\theta$ be a pseudo-Hermitian form on $\mathbb{M}$ with respect to which the Levi form $L_\theta$ is positive definite and thus defines a Riemannian metric $g$ on $\M$ (the Webster metric). The kernel of $\theta$ defines \textbf{a horizontal}  bundle $\mathcal H$. The triple $(\M, \mathcal H,g)$ is a sub-Riemannian manifold. The CR structure on $\M$ is said to be Sasakian if the Reeb vector field of $\theta$ is a sub-Riemannian Killing vector field. Let us denote by $\mathbf{Ric}_\nabla$ the Ricci curvature tensor of the Tanaka-Webster connection on $\M$. In this paper, we prove the following global version of Theorems \ref{T:comparison} and \ref{T:doubling}
\begin{theorem}
Let $\M$ be a complete Sasakian manifold. Let us assume that  there exists $K \in \mathbb{R}$, such that for every $V \in \mathcal{H}$, $$\mathbf{Ric}_\nabla (V,V) \ge -K \| V \|^2,$$ then:
\begin{enumerate}
\item (Distance comparison theorem) There exists a constant $C=C(n,K)>0$, such that for every $x,y \in \M$,
\[
d(x,y)\le C \max \{ d_R (x,y), \sqrt{d_R (x,y)} \}.
\]
\item (Uniform local volume doubling property) For every $R>0$, there exists a  constant $C=C(R,n,K)>0$ such that for any $x\in \bM$, 
with $Q= \log_2 C$ one has
\[
\mu(B(x,tr)) \ge C^{-1} t^{Q} \mu(B(x,r)),\ \ \ 0\le t\le 1,\
0<r\le R.
\]
\end{enumerate} 
\end{theorem}

The dependency of the constant $C$ on $R$ in the volume estimate is described more precisely in Theorem \ref{th-doubling}.

The method we use to approach these types of results are heat equation techniques and sharp Gaussian bounds for the heat kernel relying on the methods developed in \cite{BBG} and \cite{BG1}.
In particular, we find it convenient to work in the context of a local Dirichlet space associated to a subelliptic diffusion  operator.
This \textit{abstract} presentation has the advantage to encompass in the same framework many relevant  examples of different nature. 

\

The paper is organized as follows. In Section 2, we introduce the framework  of \cite{BG1} and recall  the generalized curvature dimension inequality that is going to be our main device in this paper. 
In Section 3, we study sharp Harnack inequalities for solutions of the sub-Riemannian heat equations. The main novelty here with respect to \cite{BG1} and \cite{BBG} is that these Harnack inequalities involve a family of distances that interpolate between the sub-Riemannian distance and the Riemannian one. Section 4 is devoted to the proof of the uniform volume doubling property. We skip most of the details in some of the proofs since the methods are close to the methods of \cite{BBG}. However, due to the more general setting, several computations are more involved.  In Section 5, we establish through sharp upper Gaussian bounds for the heat kernel the distance comparison theorem. Section 6 shows how  the distance comparison theorem is used to prove  subelliptic estimates. In particular, the fact that the sub-Riemannian distance behaves as $\sqrt{d_R (x,y)}$ for close $x,y$ implies that the operator $L$ is subelliptic of order $1/2$. Finally Section 7 establishes a sub-Riemannian Gromov  type precompactness  theorem which is obtained as a consequence of our volume estimates.

\

\textbf{Acknowledgements:} The authors would like to thank an anonymous referee for her/his very careful reading and  constructive remarks.

\section{The sub-Riemannian curvature dimension inequality and main results}

We consider a measure metric space $(\M,d,\mu)$, where $\M$ is $C^\infty$ connected manifold endowed with a  $C^\infty$ measure $\mu$, and $d$ is a metric canonically associated with a $C^\infty$ second-order diffusion operator $L$ on $\M$ with real coefficients. We assume that $L$ is locally subelliptic on $\M$ in the sense of \cite{FP}, and that moreover: 
\begin{itemize}
\item[(i)] $L1=0$; 
\item[(ii)]
$\int_\bM f L g d\mu=\int_\bM g Lf d\mu$;
\item[(iii)] $\int_\bM f L f d\mu \le 0$,
\end{itemize}
for every $f , g \in C^ \infty_0(\bM)$. The quadratic functional $\Gamma(f) = \Gamma(f,f)$, where
\begin{equation}\label{gamma}
\Gamma(f,g) =\frac{1}{2}(L(fg)-fLg-gLf), \quad f,g \in C^\infty(\bM),
\end{equation}
is known as \textit{le carr\'e du champ}. Notice that $\Gamma(f) \ge 0$ and that $\Gamma(1) = 0$.

An absolutely continuous curve $\gamma: [0,T] \rightarrow \bM$ is said to be subunit for the operator $\Gamma$ if for every smooth function $f : \bM \to \mathbb{R}$ we have $ \left| \frac{d}{dt} f ( \gamma(t) ) \right| \le \sqrt{ (\Gamma f) (\gamma(t)) }$.  We then define the subunit length of $\gamma$ as $\ell_s(\gamma) = T$. Given $x, y\in \M$, we indicate with 
\[
S(x,y) =\{\gamma:[0,T]\to \M\mid \gamma\ \text{is subunit for}\ \Gamma, \gamma(0) = x,\ \gamma(T) = y\}.
\]
In this paper we assume that 
\[
S(x,y) \not= \varnothing,\ \ \ \ \text{for every}\ x, y\in \M.
\]
Under such assumption  it is easy to verify that
\begin{equation}\label{ds}
d(x,y) = \inf\{\ell_s(\gamma)\mid \gamma\in S(x,y)\},
\end{equation}
defines a true distance on $\M$. Furthermore, it is known that
\begin{equation}\label{di}
d(x,y)=\sup \left\{ |f(x) -f(y) | \mid f \in  C^\infty(\bM) , \| \Gamma(f) \|_\infty \le 1 \right\},\ \ \  \ x,y \in \bM.
\end{equation}
Throughout this paper we assume that the metric space $(\M,d)$ is complete.

\

In addition to $\Gamma$, we assume that there exists another first-order bilinear form $\Gamma^{Z}$ satisfying for $f,g,h\in C^{\infty}\left(\mathbb{M}\right)$:
\begin{itemize}
\item[1)] $\Gamma^{Z}\left(fg,h\right)=f\Gamma^{Z}\left(g,h\right)+g\Gamma^{Z}\left(f,h\right)$;
\item[2)] $\Gamma^Z(f) = \Gamma^{Z}\left(f,f\right)\geq 0$.
\end{itemize}

% We make the important assumption that the following identity holds:
% \begin{equation}\label{G-GZ}
% \forall f\in \mathcal C^\infty(\M), \Gamma(f,\Gamma^Z(f,f))= \Gamma^Z(f,\Gamma(f,f)).
% \end{equation}

We introduce the following second-order differential forms: 
\begin{displaymath}
\Gamma_{2}\left(f,g\right)=\frac{1}{2}\left[L\Gamma\left(f,g\right)-\Gamma\left(f,Lg\right)-\Gamma\left(g,Lf\right)\right],
\end{displaymath}
\begin{displaymath}
\Gamma^{Z}_{2}\left(f,g\right)=\frac{1}{2}\left[L\Gamma^{Z}\left(f,g\right)-\Gamma^{Z}\left(f,Lg\right)-\Gamma^{Z}\left(g,Lf\right)\right],
\end{displaymath}
and we let $\Gamma_2(f) = \Gamma_2(f,f)$, $\Gamma_2^Z(f) = \Gamma^Z_2(f,f)$. \\

We also introduce a family  of control distances $d_\tau$ for $\tau \ge 0$. Given $x, y\in \M$, let us consider 
\[
S_\tau(x,y) =\{\gamma:[0,T]\to \M\mid \gamma\ \text{is subunit for}\ \Gamma+\tau^2 \Gamma^Z, \gamma(0) = x,\ \gamma(T) = y\}.
\]
A curve which is subunit for $\Gamma$ is obviously subunit for $ \Gamma+\tau^2 \Gamma^Z$, therefore $S_\tau(x,y) \neq \emptyset$. We can then define
\begin{equation}
\label{distance-riem}
d_\tau(x,y) = \inf\{\ell_s(\gamma)\mid \gamma\in S_\tau(x,y)\}.
\end{equation} 
Note that $d(x,y)= d_0(x,y)$ and that, clearly: $d_\tau(x,y)\leq d(x,y)$.

%For example, in the case of the Grushin plane, $d_\tau$ corresponds to the "Riemannian" distance where the metric is defined by 
%$$
%g^\tau_{x,y} (u,v)= u^2 + min(\tau^2, x^2 )v^2. 
%$$
%Note that in this case, the metric is not a $\mathcal C^2$ function.

\noindent The following definition was introduced in \cite{BG1}.

\begin{definition}\label{D:cdi}
We shall say that  $\mathbb{M}$ satisfies the generalized curvature-dimension inequality CD($\rho_{1},\rho_{2},\kappa,d$) if there exist constants $\rho_{1}\in\mathbb{R}, \rho_{2}>0, \kappa\geq 0,$ and $d > 0$ such that the inequality
\begin{equation}
\label{sRCD}
\Gamma_{2}\left(f\right)+\nu\Gamma^{Z}_{2}\left(f\right)\geq \frac{1}{d}\left(Lf\right)^{2}+\left(\rho_{1}-\frac{\kappa}{\nu}\right)\Gamma\left(f\right)+\rho_{2}\Gamma^{Z}\left(f\right)
\end{equation}
holds for every $f\in C^{\infty}\left(\mathbb{M}\right)$ and every $\nu>0$.
\end{definition}

$\mathbb d$

Let us observe right-away that if $\rho_1'\ge \rho_1$, then $CD(\rho_1',\rho_2,\kappa,d) \Longrightarrow$\ CD$(\rho_1,\rho_2,\kappa,d)$. To provide the reader with some perspective on Definition \ref{D:cdi} we refer the reader to \cite{BG1} but point out that it constitutes a generalization of the so-called \emph{curvature-dimension inequality} CD$(\rho_1,n)$ from Riemannian geometry. We recall that the latter is said to hold on a $n$-dimensional Riemannian manifold $\M$ with Laplacian $\Delta$ if there exists $\rho_1\in \R$ such that for every $f\in C^\infty(\M)$ one has
\begin{equation}\label{cdr}
\Gamma_2(f) \ge \frac 1n (\Delta f)^2 + \rho_1 |\nabla f|^2,
\end{equation}
where 
\[
\Gamma_2(f) =  \frac 12 \big(\Delta |\nabla f|^2 - 2 <\nabla f,\nabla(\Delta f)>\big).
\]
To see that \eqref{sRCD} contains \eqref{cdr} it is enough to take $L = \Delta$, $\Gamma^Z = 0$, $\kappa = 0$, and $d = n$, and notice that \eqref{gamma} gives $\Gamma(f) = |\nabla f|^2$ (also note that in this context the distance \eqref{di} is simply the Riemannian distance on $\M$). 
It is worth emphasizing at this moment that, remarkably, on a complete Riemannian manifold the inequality \eqref{cdr} is equivalent to the lower bound Ric $\ge \rho_1$. 

The essential new aspect of the generalized curvature-dimension inequality CD$(\rho_1,\rho_2,\kappa,d)$ with respect to the Riemannian inequality CD$(\rho_1,n)$ in \eqref{cdr} is the presence of the a priori non-intrinsic forms $\Gamma^Z$ and $\Gamma^Z_2$. In the non-Riemannian framework of this paper the form $\Gamma$ plays the role of the square of the length of a gradient along the (horizontal) directions canonically associated with the operator $L$, whereas the form $\Gamma^Z$ should be thought of as the square of the length of a gradient in the missing (vertical) directions.  

In Definition \ref{D:cdi} the parameter $\rho_1$ plays a special role. For the results in this paper such parameter represents the lower bound on a sub-Riemannian generalization of the Ricci tensor. 
The case when $\rho_1 \ge 0$ is, in our framework, the counterpart of  the Riemannian Ric\ $\ge 0$. For this reason, when in this paper we say that $\M$ satisfies the curvature dimension inequality CD$(\rho_1,\rho_2,\kappa,d)$ with $\rho_1\ge 0$, we will routinely avoid repeating at each occurrence the sentence ``for some $\rho_2>0$, $\kappa\ge 0$ and $d>0$''. 

%The role of the parameters $\rho_2, \kappa$ and $d$ will be explained by the discussion of the examples in Section \ref{S:examples}. 

Before stating our main result we need to introduce  further technical assumptions on the forms $\Gamma$ and $\Gamma^Z$:
\begin{itemize}
\item[(H.1)] There exists an increasing
sequence $h_k\in C^\infty_0(\bM)$   such that $h_k\nearrow 1$ on
$\bM$, and \[
||\Gamma (h_k)||_{\infty} +||\Gamma^Z (h_k)||_{\infty}  \to 0,\ \ \text{as} \ k\to \infty.
\]
\item[(H.2)]  
For any $f \in C^\infty(\bM)$ one has
\[
\Gamma(f, \Gamma^Z(f))=\Gamma^Z( f, \Gamma(f)).
\]
\item[(H.3)]  The heat semigroup generated by $L$, which will  be  denoted by $P_t$ throughout the paper, is stochastically complete that is, for $t \ge 0$, $P_t 1=1$, and   for every $f \in C_0^\infty(\bM)$ and $T \ge 0$, one has 
\[
\sup_{t \in [0,T]} \| \Gamma(P_t f)  \|_{ \infty}+\| \Gamma^Z(P_t f) \|_{ \infty} < +\infty.
\]

\end{itemize}

The hypothesis (H.1) and (H.2) will be in force throughout the paper. 
Let us notice explicitly that when $\bM$ is a complete Riemannian manifold with $L = \Delta$, then (H.1)  and (H.2) are fulfilled. In fact, (H.2) is trivially satisfied since we can take $\Gamma^Z \equiv 0$, whereas (H.1) follows from (and it is in fact equivalent to) the completeness of $(\M,d)$. Actually, more generally, in the geometric examples encompassed by the framework of this paper (as we have said before, for a detailed discussion of these examples the reader should consult the preceding paper \cite{BG1}), (H.1) is equivalent to assuming that $(\M,d)$ be a complete metric space. The reason is that in those examples $\Gamma+\Gamma^Z$ is the \textit{carr\'e du champ} of the Laplace-Beltrami of a Riemannian structure whose completeness is equivalent to the completeness of  $(\M,d)$. The hypothesis (H.3) has been shown in \cite{BG1}  to be a consequence of the curvature-dimension inequality  CD$(\rho_1,\rho_2,\kappa,d)$ only in many examples. 
We mention the following results from \cite{BG1}.  

\begin{theorem}\label{T:sasakiani}
Let $(\bM,\theta)$ be a complete \emph{CR} manifold  with real dimension $2n+1$ and vanishing Tanaka-Webster torsion, i.e., a Sasakian manifold. If 
for every $x\in \bM$ the Tanaka-Webster Ricci tensor satisfies the bound  
\[
\emph{Ric}_x(v,v)\ \ge \rho_1|v|^2,
\]
for every horizontal vector $v\in \mathcal H_x$,
then, for the \emph{CR} sub-Laplacian of $\bM$  the curvature-dimension inequality \emph{CD}$(\rho_1,\frac{d}{4},1,d)$ holds, with $d = 2n$. Furthermore, the Hypothesis  (H.1), (H.2) and (H.3) are satisfied.
\end{theorem}

\begin{theorem}\label{T:carnotCD}
Let $\bG$ be a Carnot group of step two, with $d$ being the dimension of the horizontal layer of its Lie algebra.
Then, $\bG$ satisfies the generalized curvature-dimension inequality \emph{CD}$(0,\rho_2, \kappa, d)$ (with respect to any sub-Laplacian $L$ on $\bG$),  
where $\rho_2$  and $\kappa$ are appropriately (and explicitly) determined in terms of the group constants. Moreover, the Hypothesis  (H.1), (H.2) and (H.3) are satisfied. 
\end{theorem}

Theorem \ref{T:carnotCD} says, in particular, that in our framework, every Carnot group of step two is a \emph{sub-Riemannian manifold with nonnegative Ricci curvature}.
CR Sasakian manifolds and Carnot groups of step two are included in, but do not exhaust, the class of sub-Riemannian manifolds with transverse symmetries of Yang-Mills type. Such wide class was extensively analyzed in Section 2 of \cite{BG1}, and we refer to that source for the relevant notions. In view of Theorems \ref{T:sasakiani} and Theorem \ref{T:carnotCD} it should be clear that our approach allows for the first time to extend the Li-Yau program, and many of its fundamental consequences, to situations which are genuinely non-Riemannian. 
As a further comment, we mention that, if we assume that the generalized curvature-dimension inequality CD$(\rho_1,\rho_2,\kappa,d)$  is  satisfied,  then  the assumption (H.3) should not be seen as restrictive.
As we mentioned above, it was shown in \cite{BG1} that (H.3) is fulfilled for all sub-
Riemannian manifolds with transverse symmetries of Yang-Mills type.

We are now ready to state the main results of this paper.  

\begin{theorem}\label{T:main}
Suppose that the generalized curvature-dimension inequality hold for some $\rho_1\in\R$.
Then, there exist constants $C_1, C_2>0$, depending only on $\rho_1, \rho_2, \kappa, d$, for which one has for every $x,y\in \M$ and every $r>0$:
\begin{equation}\label{dcsr}
\mu(B(x,2r)) \le \ C_{1}\exp\left(C_{2}r^{2}\right)\mu(B(x,r));
\end{equation}

\end{theorem}

\begin{theorem}\label{T:main2}
Suppose that the generalized curvature-dimension inequality hold for some $\rho_1\in\R$. Let $\tau \ge 0$. Then, there  exists a constant $C(\tau)>0$, depending only on $\rho_1, \rho_2, \kappa, d$ and $\tau$ for which one has for every $x,y\in \M$:
\begin{equation}\label{NSWimp}
d\left(x,y\right)\leq C(\tau) \max\{ \sqrt{ d_{\tau}\left(x,y\right)},d_{\tau}\left(x,y\right)\}.
\end{equation}
\end{theorem}

\section{Harnack inequalities}

In the sequel of the paper, we assume that besides the assumptions specified in the previous \textbf{section} the generalized curvature dimension of  Definition \ref{D:cdi}  is satisfied for some parameters $\rho_1, \rho_2, \kappa,d$.

We will denote

\begin{equation}
\label{D}
D=d\left(1+\frac{3\kappa}{2\rho_{2}}\right).
\end{equation}

The main tool to prove the  fore mentioned theorems, is the heat semigroup $P_t=e^{tL}$, which is defined using the spectral theorem.  Thanks to the hypoellipticity of $L$, for $f \in L^p(\bM)$,  $1 \le p \le \infty$, the function $(t,x) \rightarrow P_t f(x)$ is
smooth on $\mathbb{M}\times (0,\infty) $ and
\[ P_t f(x)  = \int_{\mathbb M} p(x,y,t) f(y) d\mu(y)\] where $p(x,y,t) = p(y,x,t) > 0$ is the so-called heat
kernel associated to $P_t$.

It was proved in \cite{BG1} that the generalized curvature dimension inequality implies a Li-Yau type estimate for the heat semigroup. More precisely, let  $f\geq 0$, be a non zero smooth and compactly supported function then the following inequality holds for $t>0$:
\begin{align}\label{Li-Yau}
\Gamma\left(\ln P_{t}f\right)+\frac{2\rho_{2}}{3}t\Gamma^{Z}\left(\ln P_{t}f\right)  
\leq \left(\frac{D}{d}+\frac{2\rho^-_{1}}{3}t\right)\frac{LP_{t}f}{P_{t}f}+\frac{d(\rho_1^-)^{2}}{6}t + \frac{\rho^-_1 D}{2} +\frac{D^{2}}{2dt},
\end{align}
where $\rho_1^-=\max (-\rho_1,0)$.
A consequence of the Li-Yau inequality is the parabolic Harnack inequality for the heat semigroup as it was established in \cite{BG1}. The distance used in \cite{BG1} to control the heat kernel is the sub-Riemannian distance $d$. In this section our purpose is to take advantage of the upper bound on  $\Gamma^{Z}\left(\ln P_{t}f\right) $ that is also provided by the Li-Yau inequality in order  to deduce a control of the heat kernel by a family of Riemannian distances. This control of the heat kernel in terms of Riemannian distances is the key point to prove the distance comparison theorem.

As a first step, we observe that as a straightforward consequence of \eqref{Li-Yau} we obtain  that for every $\tau \geq 0$ and $t >0$,
\begin{align}\label{Li-Yau-Riem}
\left(1+\frac{3\tau^2}{2\rho_{2}t}\right)^{-1}\left(\Gamma\left(\ln P_{t}f\right)+\tau^2 \Gamma^{Z}\left(\ln P_{t}f\right)\right)
\leq  \left(\frac{D}{d}+\frac{2\rho^-_1}{3}t\right)\frac{LP_{t}f}{P_{t}f}+\frac{d(\rho_1^-)^{2}}{6}t
 + \frac{\rho^-_1 D}{2} +\frac{D^{2}}{2dt}.
\end{align}

\begin{theorem}[Harnack inequality]\label{T:harnack}
Let $f\in C_b^{\infty}\left(\mathbb{M}\right)$ be such that $f\ge 0$, and consider $v\left(x,t\right)=P_{t}f\left(x\right)$. For every $\left(x,s\right),\left(y,t\right) \in \mathbb{M}\times \left(0,\infty\right)$ with $s<t$ one has with $D$ as in \eqref{D}
\begin{equation}\label{harnack}
\frac{v\left(x,s\right)}
{v\left(y,t\right)} \leq \left(\frac{t}{s}\right)^{\frac{D}{2}} \exp\left(\frac{ d\rho_1^-\left(t-s\right)}{4} \right)
\exp\left(\frac{d_\tau\left(x,y\right)^{2}}{4\left(t-s\right)}
\left( \left(\frac{D}{d}+ \tau^2 \frac{\rho_1^-}{\rho_2} \right) +\frac{\rho_1^-}{3}(t+s) + \frac{3\tau^2 D}{2(t-s)\rho_2d}  \ln \left(\frac{t}{s}\right)\right)\right).
\end{equation}
\end{theorem}

\begin{proof}
We can assume $\rho_1 \le 0$. Otherwise, if $\rho_1 >0$ then CD($0,\rho_{2},\kappa,d$) anyhow also holds.  We can rewrite the Li-Yau type inequality in the form
\begin{equation}\label{Li-Yau2}
\Gamma\left(\ln P_{u}f\right)+\tau^2 \Gamma^{Z}\left(\ln P_{u}f\right) \leq a_\tau(u) \frac{LP_u f}{P_u f} + b_\tau(u)
\end{equation}
where $$a_\tau(u)= \left(1+\frac{3\tau^2}{2\rho_{2}u}\right) \left(\frac{D}{d}+\frac{2\rho_1^-}{3}u\right)$$
and $$b_\tau(u)= \left(1+\frac{3\tau^2}{2\rho_{2}u}\right)  \left(\frac{d(\rho_1^-)^{2}}{6}u+\frac{\rho_1^-D}{2}  +\frac{D^{2}}{2du}\right).$$
Let now  $x,y \in \mathbb M$ and let $\sigma: [0,T] \to \bM$ be a subunit curve for $\Gamma+\tau^2 \Gamma^Z$ such that $\sigma(0)=x, \sigma(T)=y$. For $s \le u \le t$, we denote
\[
\gamma(u)=\sigma \left( \frac{u-s}{t-s} T \right).
\]

 Let us now consider
 \[
 \phi(u)=\ln P_u(f)(\gamma(u)).
 \]

 We  compute
 \[
  \phi'(u)=\frac{1}{P_u f(\gamma(u))} \left( LP_u f (\gamma(u)) + \frac{d}{du} \left( P_u f(\gamma(u)) \right)\right).
 \]
 Since $\sigma$ is subunit for $\Gamma+\tau^2 \Gamma^Z$, we have
 \[
  \frac{d}{du} \left( P_u f(\gamma(u)) )\right) \ge -\frac{T}{t-s} \sqrt{\Gamma(P_uf)(\gamma(u)) + \tau^2 \Gamma^Z(P_uf)(\gamma(u)) } 
 \]
 Now, for every $\lambda >0$, we have
 \[
  \sqrt{\Gamma(P_uf)(\gamma(u)) + \tau^2 \Gamma^Z(P_uf)(\gamma(u)) }  \le \frac{1}{2 \lambda} +\frac{\lambda}{2}  \left(\Gamma(P_uf)(\gamma(u)) + \tau^2 \Gamma^Z(P_uf)(\gamma(u)) \right).
  \]
 Therefore we obtain
 \begin{align*}
   \phi'(u) & \ge \frac{1}{P_u f(\gamma(u))} \left( LP_u f (\gamma(u)) -\frac{T}{t-s} \left( \frac{1}{2 \lambda} +\frac{\lambda}{2}  \left(\Gamma(P_uf)(\gamma(u)) + \tau^2 \Gamma^Z(P_uf)(\gamma(u)) \right) \right)   \right) \\
    & \ge  \frac{1}{P_u f(\gamma(u))} \left( LP_u f (\gamma(u)) -\frac{T}{t-s} \left( \frac{1}{2 \lambda} +\frac{\lambda}{2}  \left( a_\tau(u) (LP_u f)(\gamma(u))(P_u f)(\gamma(u)) + b_\tau(u)(P_uf)^2(\gamma(u)) \right) \right)   \right) 
 \end{align*}
 Choosing $\lambda=\frac{2(t-s)}{T a_\tau(u)  P_u f(\gamma(u)) }$ yields
 \[
  \phi'(u)  \ge -\frac{a_\tau(u) T^2}{4(t-s)^2} -\frac{b_\tau(u)}{a_\tau(u)}.
 \]
By integrating this inequality from $s$ to $t$ we infer
\[
\ln P_t f (y) -\ln P_s f(x) \ge  -\frac{\int_s^t a_\tau(u) du }{4(t-s)^2} T^2-\int_s^t \frac{b_\tau(u)}{a_\tau(u)}du.
\]
Minimizing over sub-unit curves gives
\[
\ln P_t f (y) -\ln P_s f(x) \ge  -\frac{\int_s^t a_\tau(u) du }{4(t-s)^2} d_\tau  (x,y)^2-\int_s^t \frac{b_\tau(u)}{a_\tau(u)}du,
\]
which is the claimed result after tedious computations.
\end{proof} 

\begin{remark}
This result was already established in  \cite{IM} for the Carnot-Carath\'edory distance ($\tau=0$). 
\end{remark}

%\begin{theorem}[Harnack inequality]\label{T:harnack}
%Assume that \eqref{sRCD} holds with $\rho_{1}<0$. Let $f\in C_b^{\infty}\left(\mathbb{M}\right)$ be such that $f\ge 0$, and consider $u\left(x,t\right)=P_{t}f\left(x\right)$. For every $\left(x,s\right),\left(y,t\right) \in \mathbb{M}\times \left(0,\infty\right)$ with $s<t$ one has with $D$ as in \eqref{D}
%\begin{equation}\label{harnack}
%u\left(x,s\right)\leq u\left(y,t\right)\left(\frac{t}{s}\right)^{\frac{D}{2}}\exp\left(\frac{d\left(x,y\right)^{2}}{4\left(t-s\right)}\left(\frac{D}{d} +\frac{2|\rho_{1}|}{3}t\right)+\frac{3d|\rho_{1}|\left(t-s\right)}{4}\right).
%\end{equation}
%\end{theorem}

This inequality can easily  be extended to the heat kernel.

\begin{corollary}
\label{Harnack kernel}
Let $p\left(x,y,t\right)$ be the heat kernel on $\mathbb{M}$. For every $x,y,z\in \mathbb{M}$, every $0\leq s\leq t<\infty$ and every $\tau\geq 0$, one has
\begin{displaymath}
\frac{p\left(x,y,s\right)}
{p\left(x,z,t\right)} \leq \left(\frac{t}{s}\right)^{\frac{D}{2}} \exp\left(\frac{d\rho_1^-\left(t-s\right)}{4}\right)
\exp\left(\frac{d_\tau\left(x,y\right)^{2}}{4\left(t-s\right)}
\left( \left(\frac{D}{d}+ \tau^2 \frac{\rho_1^-}{\rho_2} \right) +\frac{\rho_1^-}{3}(t+s) + \frac{3\tau^2 D}{2(t-s)\rho_2d}  \ln \left(\frac{t}{s}\right)\right)\right).
\end{displaymath}
\end{corollary}

The following proposition provides a pointwise estimate of the volume of the balls, for the proof we refer the reader to \cite{IM}. 

\begin{proposition}\label{RhoNeg1}
There exists a constant $C\left(d,\kappa,\rho_{2}\right)>0$ such that, given $R_{0}>0$, for every $x\in\mathbb{M}$ and every $R\geq R_{0}$ one has
\begin{displaymath}
\mu\left(B\left(x,R\right)\right)\leq C\left(d,\kappa,\rho_{2}\right) \frac{\exp\left(2d\rho_1^- R^{2}_{0}\right)}{R^{D}_{0}p\left(x,x,R^{2}_{0}\right)}R^{D}\exp\left(2d\rho_1^- R^{2}\right).
\end{displaymath}
\end{proposition}

\begin{remark}
In the above proposition, the exponential square volume growth  \textbf{may not be} optimal, see for instance \cite{BW} where an exponential growth is proved on contact manifolds.
\end{remark}

\section{Volume doubling property}

We now turn to the proof of Theorem \ref{T:main}. Though, some new ideas and more careful estimates are required, the proof  mainly follows the lines of \cite{BBG} where the results is proved when $\rho_1=0$. Therefore, several results are stated without proof and we only justify the statements involving these new ideas and careful estimates. The results given without justification may be proved as in \cite{BBG} by keeping track of the term $\rho_1^-$.

Henceforth in the sequel we denote
\[
C^\infty_b(\bM) = C^\infty(\bM) \cap L^\infty(\bM).
\]
For $\varepsilon>0$ we also denote by $\mathcal{A}_\varepsilon$ the set of functions $f \in C^\infty_b(\M)$ such that
\[
f=g+\varepsilon,
\]
for some $\varepsilon >0$ and some $g \in C^\infty_b(\M)$, $g \ge 0$, such that $g, \sqrt{\Gamma(g)}, \sqrt{\Gamma^Z(g)} \in L^2(\bM)$. As shown in \cite{BG1}, this set is stable under the action of $P_t$, i.e.,  if  $f\in \mathcal{A}_\varepsilon$, then $P_t f \in  \mathcal{A}_\varepsilon$.

The first ingredient to prove the doubling property is the following reverse log-Sobolev inequality.

\begin{theorem}
\label{Rev-Log}
Let $\ve>0$ and $f\in\mathcal{A}_{\ve}$, then for every $C\geq 0$, one has for $x\in\mathbb{M}$,$t>0$,
\begin{align*}
& \frac{t}{\rho_{2}}P_{t}f\left(x\right)\Gamma\left(\ln P_{t}f\right)\left(x\right)+t^{2}P_{t}f\left(x\right)\Gamma^{Z}\left(\ln P_{t}f\right)\left(x\right)\\
\leq&\frac{1}{\rho_{2}}\left(1+\frac{2\kappa}{\rho_{2}}+\frac{4C}{d}+2t\rho_1^-\right)\left[P_{t}\left(f\ln f\right)\left(x\right)-P_{t}f\left(x\right)\ln P_{t}f\left(x\right)\right]
-\frac{4C}{d \rho_{2}}\frac{t}{1+\delta}LP_{t}f\left(x\right)+\frac{2C^{2}}{d\rho_{2}}\ln\left(1+\frac{1}{\delta}\right)P_{t}f\left(x\right).
\end{align*}

\end{theorem}

This inequality admits the following corollary.

\begin{proposition}
Let $\ve>0$, $f\in\mathcal{A}_{\ve}$ such that $\ve\leq f\leq 1$ and consider the function $u\left(x,t\right)=\sqrt{-\ln P_{t}f\left(x\right)}$. Then,
\begin{displaymath}
2tu_{t}+\left(u+\left(1+\sqrt{\frac{D^{*}}{2}}\right)u^{1/3}+\sqrt{\frac{D^{*}}{2}}u^{-1/3}\right)\left(1+\sqrt{d\rho_1^- t}\right)\geq 0,
\end{displaymath}
where
\begin{displaymath}
D^{*}=d\left(1+\frac{2\kappa}{\rho_{2}}\right).
\end{displaymath}
\label{ReverseH}
\end{proposition}

\noindent Introduce the function $g:\left(0,\infty\right)\rightarrow\left(0,\infty\right)$ defined by 
\begin{displaymath}
g\left(v\right)=\frac{1}{v+\left(1+\sqrt{\frac{D^{*}}{2}}\right)v^{1/3}+\sqrt{\frac{D^{*}}{2}}v^{-1/3}}
\end{displaymath}
Note that $g$ verifies
\begin{displaymath}
\underset{v\rightarrow 0^{+}}{\lim}\sqrt{\frac{D^{*}}{2}}v^{-1/3}g\left(v\right)=1, \quad \underset{v\rightarrow + \infty }{\lim}vg\left(v\right)=1.
\end{displaymath}
Therefore, we have $g\in L^{1}\left(0,A\right)$ for every $A>0$, but $g\notin L^{1}(0,+\infty)$. If we set
\begin{displaymath}
G\left(u\right)=\int^{u}_{0}{g\left(v\right)dv},
\end{displaymath}
then $G'\left(u\right)=g\left(u\right)>0$, and thus $G:\left(0,\infty\right)\rightarrow\left(0,\infty\right)$ is invertible. Furthermore, we can write
\begin{equation}
\label{Growth G}
G(u)=\ln \left(u\right)+C_{0}+R\left(u\right), \; u>0
\end{equation}
where $C_{0}$ is a constant and $R:(0,+\infty) \to \R$ a function such that $\underset{u\rightarrow \infty}{\lim}R\left(u\right)=0$. Proposition \ref{ReverseH} can be re-written in terms of $g$ as follows
\begin{displaymath}
2tu_{t}+\frac{1+\sqrt{td \rho_1^-}}{g\left(u\right)}\geq 0.
\end{displaymath}
Since $g\left(u\right)=G'\left(u\right)$, we conclude
\begin{equation}
\label{DerG}
\frac{dG\left(u\right)}{dt}=G'\left(u\right)u_{t}\geq -\frac{1}{2t}-\frac{1}{2}\sqrt{\frac{d\rho_1^-}{t}}.
\end{equation} 
Integrating this differential inequality leads to the following result:
\begin{corollary}
\label{Int-G}
Let $f\in L^{\infty}\left(\mathbb{M}\right)$, $0\leq f\leq 1$, then for any $x\in\mathbb{M}$ and $0<s<t$,
\begin{displaymath}
G\left(\sqrt{-\ln P_{t}f\left(x\right)}\right)\geq G\left(\sqrt{-\ln P_{s}f\left(x\right)}\right)-\frac{1}{2}\ln\left(\frac{t}{s}\right)-\sqrt{d\rho_1^-}\left(\sqrt{t}-\sqrt{s}\right).
\end{displaymath}
\end{corollary}

The second ingredient in our proof is the following  small time asymptotics.

\begin{proposition}
\label{S-T-As}
Given $x\in\mathbb{M}$  and $r>0$, let $f=\textbf{1}_{B\left(x,r\right)^{c}}$. One has,
\begin{displaymath}
\underset{s\rightarrow 0^{+}}{\lim \inf}\left(-s\ln P_{s}f\left(x\right)\right)\geq \frac{r^{2}}{4}.
\end{displaymath}
\end{proposition}

We are now ready for the following estimate.

\begin{proposition}
\label{Lower Bound G}
Let $x\in\mathbb{M}$ and $r>0$ be arbitrarily fixed. There exists a constant $C^{*}_{0}\in\mathbb{R}$ independent of $x$ and $r$, such that for any $t>0$,
\begin{displaymath}
G\left(\sqrt{-\ln P_{t}\textbf{1}_{B\left(x,r\right)^{c}}\left(x\right)}\right)\geq \ln\frac{r}{\sqrt{t}}+C^{*}_{0}-\sqrt{d\rho_1^- t}
\end{displaymath}
\end{proposition}
\begin{proof}
Let $f=\textbf{1}_{B\left(x,r\right)^{c}}$.  Corollary \ref{Int-G} and  \eqref{Growth G} give
\begin{align*}
G\left(\sqrt{-\ln P_{t}f\left(x\right)}\right)&\geq G\left(\sqrt{-\ln P_{s}f\left(x\right)}\right)+\ln\sqrt{s}-\ln\sqrt{t}-\sqrt{d\rho_1^-}\left(\sqrt{t}-\sqrt{s}\right)\\
                                              & =  \ln\sqrt{-s\ln P_{s}f\left(x\right)}+C_{0}+R\left(\sqrt{-\ln P_{s}f\left(x\right)}\right)-\ln\sqrt{t}-\sqrt{d\rho_1^-t} + \sqrt{d\rho_1^-s}.
\end{align*}
Since $\underset{s\rightarrow 0^{+}}{\lim}\left(-\ln P_{s}f\left(x\right)\right)=\infty$, we infer $\underset{s\rightarrow 0^{+}}{\lim}R\left(\sqrt{-\ln P_{s}f\left(x\right)}\right)=0$.  Letting $s\rightarrow 0^{+}$,  Proposition \ref{S-T-As} yields
 we obtain
\begin{displaymath}
G\left(\sqrt{-\ln P_{t}f\left(x\right)}\right)\geq \ln\frac{r}{2}-\ln\sqrt{t}+C_{0}-\sqrt{d\rho_1^-t}=\ln\frac{r}{\sqrt{t}}-\sqrt{d\rho_1^- t}+C^{*}_{0},
\end{displaymath}
with $C^{*}_{0}=C_{0}-\ln 2$.
\end{proof}

The following uniform  lower bound on the heat content of balls, which is already interesting in itself, will imply the volume doubling property.

\begin{theorem}\label{t:AR}
Set $C_0^{**}=G(\sqrt {\ln 2})-C_0^*$ and for $R\geq0$, define $U(R)=\Psi_R^{-1}(C_0^{**}) $ where $\Psi_R^{-1}$ is the inverse function of
$$
\Psi_R(u)=\ln \left(\frac{1}{u}\right) - \sqrt {d\rho_1^-} \, R \, u , \; u \in (0,\infty).
$$
Then for every $x\in\mathbb{M}$ and every function $A:[0,+\infty) \to (0,\infty)$ such that $\sqrt{A(R)} \leq U(R)$, we have for $r>0$,
\begin{displaymath}
P_{A\left(r\right)r^{2}}\left(\textbf{1}_{B\left(x,r\right)}\right)\left(x\right)\geq \frac{1}{2}.
\end{displaymath}
\label{P-Low-B}
\end{theorem}
\begin{proof}
By the stochastic completeness of $\mathbb{M}$ 
\begin{displaymath}
P_{A\left(r\right)r^{2}}\left(\textbf{1}_{B\left(x,r\right)}\right)\left(x\right)=1-P_{A\left(r\right)r^{2}}\left(\textbf{1}_{B\left(x,r\right)^{c}}\right)\left(x\right)
\end{displaymath}
The desired estimate is equivalent to prove 
\begin{displaymath}
\sqrt{\ln 2}\leq \sqrt{-\ln P_{A\left(r\right)r^{2}}\left(\textbf{1}_{B\left(x,r\right)^{c}}\right)\left(x\right)} 
\end{displaymath}
or equivalently,
\begin{equation}
\label{Ineq G}
G\left(\sqrt{\ln 2}\right)\leq G\left(\sqrt{-\ln P_{A\left(r\right)r^{2}}\left(\textbf{1}_{B\left(x,r\right)^{c}}\right)\left(x\right)} \right).
\end{equation}
At this point Proposition \ref{Lower Bound G} gives 
\begin{align*}
G\left(\sqrt{-\ln P_{A\left(r\right)r^{2}}\left(\textbf{1}_{B\left(x,r\right)^{c}}\right)\left(x\right)}\right)
 &\geq \ln\left(\frac{1}{\sqrt{A\left(r\right)}}\right)+C_0^*-\sqrt{d\rho_1^-A\left(r\right)}r\\
 &\geq \ln\left(\frac{1}{U(r)}\right)+C_0^*-\sqrt{d\rho_1^-} rU(r) = G(\sqrt{ \ln2}).
\end{align*}
\end{proof}

We now give some estimates for the function $U(R)$ appearing in Theorem \ref{t:AR}.
\begin{proposition}
The function $U$ is non-increasing and satisfies, for $R\geq 0$,
\[
U(R) \geq \frac{1}{\sqrt{d\rho_1^-}\, R + e^{C_0^{**}}}.
\] 
\end{proposition}

\begin{proof}
First notice that $U(0)= e^{- C_0^{**}}$ and $U$ is positive.
Since $\Psi_R(U(R)) $ is constant, taking derivative yields:
\[
U'(R)= - \frac{\sqrt{d\rho_1^-}\,  U(R)}{ \sqrt{d\rho_1^-} \; R + \frac{1}{U(R)} } \geq - \sqrt{d\rho_1^-} \, U(R)^2.
\]  
Therefore $U$ is non-increasing and integrating the above inequality
we infer that
\[
U(R) \geq \frac{1}{\sqrt{d\rho_1^-}\; R+ U^{-1}(0)}.
\]
\end{proof}

Henceforth,  in the sequel, for $r\geq 0$,  we set 
\begin{equation}\label{e:AR}
A(r)=\min(U(r)^2,1) \ge \min\left( \left( \frac{1}{\sqrt{d\rho_1^-} \; R + e^{C_0^{**}}}\right)^2 , 1\right).
\end{equation}
There exists a constant $C_1>0$ such that, for all $R\geq 0$,
\begin{equation}
\frac{C_1}{1+ d\rho_1^- R^2} \leq A(R) \leq 1.
\end{equation}

A first consequence of the uniform estimate we obtained are the following lower bounds for the heat kernel. Observe, and this is another main novelty with respect to \cite{BBG} that these bounds are written with respect to the distance $d_\tau$ (we recall that $d_0$ is the sub-Riemannian distance).

\begin{theorem}\label{th-low-hk}
Set $C_2= \frac{C_1}{4}$. 
For $t>0$ and $x\in \M$, then 
\begin{equation}\label{low-hk-d}
p(x,x,t) \geq \frac{\left(A\left(\sqrt \frac{t}{2} \right)\right)^\frac{D}{2}
\exp\left(-\frac{d\rho_1^- t}{4}\right)}{4 \mu\left(B\left(x,\sqrt \frac{t}{2} \right)\right)} \geq  \frac{C_2}{ \mu\left(B\left(x,\sqrt \frac{t}{2} \right)\right)}
\frac{\exp\left(-\frac{d\rho_1^- t}{4}\right)}{\left(1+\frac{d\rho_1^- t}{2}\right)^\frac{D}{2}}.
\end{equation}

As a consequence, for $x,y \in \M, t>0$ and $ \tau\geq 0$,
\begin{align}\label{low-hk-od}
p(x,y,t) \geq & \frac{\left(A\left( \frac{\sqrt t}{2} \right)\right)^\frac{D}{2}  2^{-\frac{D}{2}} \exp\left( -\frac{d\rho_1^- t}{4}\right) }{4 \mu\left(B\left(x,\frac{\sqrt t}{2} \right)\right)} 
\exp\left(-\frac{d_\tau\left(x,y\right)^{2}}{2t}\left(\frac{D}{d} +\frac{\rho_1^-}{2}t + \frac{2 \tau^2}{t} \left( \frac{\rho_1^-}{\rho_2} +\frac{3D}{2\rho_2d} \ln(2)\right)\right)\right)\\
\notag  \geq & \frac{ 2^{-\frac{D}{2}} C_2}{ \mu\left(B\left(x,\sqrt t \right)\right)} \frac{\exp\left( -\frac{d\rho_1^- t}{4}\right)}{\left(1+\frac{d\rho_1^- t}{4}\right)^\frac{D}{2}}
\exp\left(-\frac{d_\tau \left(y,x\right)^{2}}{2t}\left(\frac{D}{d} +\frac{\rho_1^-}{2}t + \frac{2 \tau^2}{t} \left( \frac{\rho_1^-}{\rho_2} +\frac{3D}{2\rho_2d} \ln \left(2 \right)\right)\right)\right).
\end{align}
\end{theorem}

\begin{proof}
 With the same notations as in Theorem \ref{P-Low-B}, for $R>0$, 
\begin{displaymath}
P_{A\left(R\right)R^{2}}\left(\textbf{1}_{B\left(x,R\right)}\right)\left(x\right)\geq \frac{1}{2}.
\end{displaymath}
Thus,
\begin{eqnarray*}
\frac{1}{4}&\leq & P_{A\left(R\right)R^{2}}\left(\textbf{1}_{B\left(x,R\right)}\right)\left(x\right)^{2}\\
           &=&\left( \int_\M p(x,y,A\left(R\right)R^{2}) \textbf{1}_{B\left(x,R\right)} d\mu(y) \right)^2\\
           &\leq & \int_\M p(x,y,A\left(R\right)R^{2})^2 d\mu(y) \int_\M  \textbf{1}_{B\left(x,R\right)} d\mu(y)\\
           &=& p(x,x,2A\left(R\right)R^{2}) \mu(B(x,R)).
%\label{Low-pmu}
\end{eqnarray*}
Now since $0<A(R)\leq 1$, Harnack inequality in Corollary \ref{Harnack kernel} gives
\begin{equation}
\label{HarnackA}
p\left(x,x,2A\left(R\right)R^{2}\right)\leq p\left(x,x, 2 R^{2}\right)\left(A\left(R\right)\right)^{-D/2}\exp\left(\frac{1}{2}d\rho_1^- R^{2}\right).
\end{equation}
Therefore, we proved
\begin{equation}
\label{low-pmu}
 p\left(x,x, 2 R^{2}\right) \geq \frac{A\left(R\right)^{D/2}\exp\left(- \frac{1}{2} d\rho_1^- R^{2}\right)}{4 \mu(B(x,R))}.
\end{equation}
The first point follows by setting $t= 2R^2$.

For the second point, we recall that Harnack inequality with the distance $d_\tau$ reads, for $t>0,s=\frac{t}{2}$,
$$
{p\left(x,x,\frac{t}{2}\right)}
 \leq  {p\left(x,y,t\right)} 2^{\frac{D}{2}} 
\exp\left(\frac{d\rho_1^- t}{8}\right)
\exp\left(\frac{d_\tau \left(y,x\right)^{2}}{2t}\left(\frac{D}{d} +\frac{\rho_1^-}{2}t + \frac{2 \tau^2}{t} \left( \frac{\rho_1^-}{\rho_2} +\frac{3D}{2\rho_2d} \ln \left(2 \right)\right)\right)\right).
$$
Using the first point the conclusion follows directly.
\end{proof}
\begin{remark}
For $\tau=0$, we thus obtain the following optimal lower bound for the heat kernel with respect to the subelliptic distance:
\[
p(x,y,t) \geq \frac{ 2^{-\frac{D}{2}} C_2} { \mu\left(B\left(x,\sqrt t \right)\right)} \frac{\exp\left( -\frac{d\rho_1^- t}{4}\right)}{\left(1+\frac{d\rho_1^- t}{4}\right)^\frac{D}{2}}
\exp\left(-\frac{d\left(y,x\right)^{2}}{2t}\left(\frac{D}{d} +\frac{\rho_1^-}{2}t\right)\right).
\]
\end{remark}

\begin{remark}
Observe that due to the symmetry of the heat kernel: $p(x,y,t)=p(y,x,t)$, we can replace $\mu(B(x,\sqrt{t}))$ by $\sqrt{ \mu(B(x,\sqrt{t}))\mu(B(y,\sqrt{t}))}$ in the lower bound estimates.
\end{remark} 

We are now in position to prove the doubling condition satisfied by the measure which generalizes the result of \cite{BBG} to the case where $\rho_1^- \neq 0$.
\begin{theorem}\label{th-doubling}
\label{NegB-G}
There exist constants $C_{3}>0$ and $C_{4}>0$ depending only on $d,\kappa$ and $\rho_{2}$ such that
for all $x\in\mathbb{M}$ and  all $R> 0$, we have
\begin{equation}\label{e:doubling}
\mu\left(B\left(x,2R\right)\right)\leq C_{3}\exp\left(C_{4}\rho_1^- R^{2} \right)\mu\left(B\left(x,R\right)\right)
\end{equation}
\end{theorem}
\begin{proof}
Due to Proposition \ref{RhoNeg1}, we have for $R> 0$

%\begin{equation}
%\label{Bound muR}
%\mu\left(B\left(x,2R\right)\right)\leq C\left(d,\kappa,\rho_{2}\right) \frac{\exp\left(16 d|\rho_{1}|R^{2}\right)}{p\left(x,x,4R^{2}\right)}.
%\end{equation}
%
%OTHER CHOICE 

\[
\mu\left(B\left(x,2R\right)\right)\leq C\left(d,\kappa,\rho_{2}\right) 2^\frac{D}{2} \frac{\exp\left(12 d\rho_1^- R^{2}\right)}{p\left(x,x,2R^{2}\right)}.
\]
Combining this inequality with the lower bound for the heat kernel (\ref{low-pmu}) gives the desired result:
\[
\mu(B(x,2R)) \leq C\left(d,\kappa,\rho_{2}\right) 2^{\frac{D}{2}+2} \exp \left( \frac{25}{2} d\rho_1^- R^{2}\right)  A\left(R\right)^{-D/2} \mu(B(x,R)).
\]

\end{proof}

%\begin{remark}
%Notice that actually the above proof shows  $C_{2}\rightarrow 0$ as $\rho_{1}\rightarrow 0$. This means that Theorem \ref{NegB-G} implies the doubling property when the manifold has non-negative curvature, justifying the name of this section.
%\end{remark}

\section{Distance comparison theorem}

In this section we prove a global version of  the celebrated Nagel-Stein-Wainger estimate, Theorem \ref{T:comparison} . We need the following optimal upper bound for the heat kernel $p\left(x,y,t\right)$. 
\begin{theorem}
\label{Upper-Bd}
For all $\ve>0$, there exist some constants $C_5,C_6>0$, depending only on $d,\rho_2,\kappa$ and $\ve>0$, such that for $t>0$ and $x,y \in \mathbb{M}$, one has
\begin{equation}\label{up-hk}
p(x,y,t) \leq \frac{C_5(\ve)}{\mu(B(x,\sqrt t))^{1/2} \mu(B(y,\sqrt t))^{1/2}} \exp(C_6(\ve) \rho_1^- t) \exp\left(-\frac{d^2(x,y)}{(4+\ve)t}\right).
\end{equation}
\end{theorem}
\begin{proof}
This proof follows the lines in Cao-Yau \cite{CY} and Baudoin-Garofalo \cite{BG1}. Let $\alpha>0$, then by the Harnack inequality in Corollary \ref{Harnack kernel} with $\tau=0$,
\begin{eqnarray*}
p(x,y,t)^2 &\leq& \frac{\left(1+\alpha\right)^{D}\exp\left(\frac{D}{2\alpha d}\right)}{\mu(B(y,\sqrt t))}\exp\left(\rho_1^- t \left(\frac{2+\alpha}{6\alpha}+\frac{d\alpha}{2}
\right)\right)
\int_{B(y,\sqrt t)} p (x,z,(1+\alpha)t)^2 d\mu(z) \\
           &=& \frac{\left(1+\alpha\right)^{D}\exp\left(\frac{D}{2\alpha d}\right)}{\mu(B(y,\sqrt t))}\exp\left(\rho_1^- t\left(\frac{2+\alpha}{6\alpha}+\frac{d\alpha}{2}
\right)\right) 
            P_{(1+\alpha)t} \left(p(x,\cdot,(1+\alpha)t) \1_{B(y,\sqrt t)} \right)(x). 
\end{eqnarray*}
\noindent Applying the Harnack inequality in Theorem \ref{T:harnack} once again, we have
\begin{eqnarray*}
& & P_{(1+\alpha)t} \left(p_{(1+\alpha)t} (x,\cdot) \1_{B(y,\sqrt t)} \right)(x)^2 = P_{(1+\alpha)t}(F_t)(x)^2\\
&\leq &  \frac{\left(1+\alpha\right)^{D}\exp\left(\frac{D}{2\alpha (\alpha+1)d}\right)}{\mu(B(x,\sqrt t))}\exp\left(\rho_1^- t \left(\frac{2+\alpha}{6\alpha}+\frac{d\alpha (\alpha+1)}{2}
\right)\right)\int_{B(x,\sqrt t)} P_{(1+\alpha)^2t} (F_t)(z)^2 d\mu(z)\\
\end{eqnarray*}
with $F_t(.)=p (x,\cdot,(1+\alpha)t) \1_{B(y,\sqrt t)}(\cdot)$.\\

By using now an argument of  the  proof of Theorem 8.1 in \cite{BG1} and the fact that  for  $z \in B(y,\sqrt t)$,
$$
d^2(x,z)\geq \frac{d^2(x,y)}{1+\alpha} -\frac{t}{\alpha},
$$
 we have for $0<(1+\alpha)^2 t<T$,
\begin{eqnarray*}
& &\int_{B(x,\sqrt t)} P_{(1+\alpha)^2t} \left(F_t \right)(z)^2 d\mu(z)\\
&\leq &  \exp\left(\frac{ t}{2(T-(1+\alpha)^2t)}\right) \int_\mathbb{M} e^{g((1+\alpha)^2t,z)} P_{(1+\alpha)^2t} \left(F_t \right)(z)^2 d\mu(z)\\
&\leq & \exp\left(\frac{ t}{2(T-(1+\alpha)^2t)}\right) \int_\mathbb{M} e^{g(0,z)} \left(F_t \right)(z)^2 d\mu(z)\\
&=& \exp\left(\frac{ t}{2(T-(1+\alpha)^2t)}\right) \int_{B(y,\sqrt t)} \exp\left(-\frac{ d^2(x,z)}{2T}\right) p_{(1+\alpha)t} (x,z)^2 d\mu(z)\\
&\leq&  \exp\left(\frac{ t}{2(T-(1+\alpha)^2t)}+ \frac{t}{2\alpha T} - \frac{ d^2(x,y)}{2(1+\alpha)T} \right)\int_{B(y,\sqrt t)}  p_{(1+\alpha)t} (x,z)^2  d\mu(z) \\
&=& \exp\left(\frac{ t}{2(T-(1+\alpha)^2t)}+ \frac{t}{2\alpha T} - \frac{ d^2(x,y)}{2(1+\alpha)T}\right) P_{(1+\alpha)t} (F_t)(x)
\end{eqnarray*}

where for $0\leq s<T$ and $z\in \M$, the function $g$ is defined by
\[
 g(s, z)= - \frac{d^2(x,z)}{2(T-s)}.
\]
 
\noindent Finally,
\begin{eqnarray*}
p(x,y,t)&\leq &  \frac{\left(1+\alpha\right)^{D}\exp\left(\frac{D(\alpha+2)}{4\alpha (\alpha+1)d}\right)}{\mu(B(x,\sqrt t))^{1/2}\mu(B(y,\sqrt t))^{1/2}}\exp\left(\rho_1^- \left(2+\alpha\right)t\left(\frac{1}{6\alpha}+\frac{d\alpha}{4}\right)\right) \\
 & & \exp\left(\frac{ t}{4(T-(1+\alpha)^2t)}+ \frac{t}{4\alpha T} - \frac{ d^2(x,y)}{4(1+\alpha)T}\right).
\end{eqnarray*}
\noindent Hence the result follows by choosing $T=(1+\alpha)^3t$.
\end{proof}
\begin{remark}
Note that $C_5(\ve)$ and $C_6(\ve)$ both tend to infinity when $\ve \to 0$.
\end{remark}

We are finally in position to prove the distance comparison theorem.

\begin{theorem}\label{th-distance} 
There exists a constant $C_7 >0$ which depends only on $d,\kappa$ and $\rho_2$ such that for all $x$ and $y$ in $\mathbb{M}$ and all $0<\tau \leq 1$,
\[
d\left(x,y\right)\leq C_7  \left(1+\sqrt {\rho_1^-} \right) \max \left\{  \sqrt{ d_{\tau}\left(x,y\right)}, d_\tau(x,y) \right\}.
\]
\end{theorem}

\begin{proof}
%If we integrate the estimate in Remark \ref{Li-Yau tau} against geodesics we obtain
%\begin{displaymath}
%p\left(x,y,s\right)\leq p\left(x,z,t\right)\left(\frac{t}{s}\right)^{\frac{D}{2}}\exp\left(\frac{d_{\tau}\left(y,z\right)^{2}}{4\left(t-s\right)}\left(\frac{D}{d}+\frac{2\left|\rho_{1}\right|}{3}t\right)\left(1+\frac{3\tau}{2\rho_{2}s}\right)+\frac{3d\left|\rho_{1}\right|\left(t-s\right)}{4}\right)s, for $t>0,s=\frac{t}{2}$,
%$$
%{p\left(x,x,\frac{t}{2}\right)}
% \leq  {p\left(x,z,t\right)} 2^{\frac{D}{2}} \exp\left(\frac{d|\rho_{1}|t}{4}\right)
%\exp\left(\frac{d_\tau\left(x,y\right)^{2}}{2t}\left(\frac{D}{d} +\frac{2|\rho_{1}|}{3}t + \tau^2 \left( \frac{|\rho_1|}{\rho_2} +\frac{3D}{2\rho_2d} \ln(2)\right)\right)\right).
%$$

Using the symmetry of the heat kernel, combining the lower estimate (\ref{low-hk-od}) for the heat kernel with the distance $d_\tau$  and the upper estimate (\ref{up-hk}) for the sub-elliptic distance $d$  gives
\begin{align*}
 & \frac{\left(A\left( \frac{\sqrt t}{2} \right)\right)^\frac{D}{2}  2^{-\frac{D}{2}} \exp\left( -\frac{d\rho_1^- t}{4}\right) }{4 \mu\left(B\left(x,\frac{\sqrt t}{2}\right) \right)^\frac{1}{2} 
 \mu\left(B\left(y,\frac{\sqrt t}{2}\right) \right)^\frac{1}{2} } 
\exp\left(-\frac{d_\tau\left(x,y\right)^{2}}{2t}\left(\frac{D}{d} +\frac{\rho_1^-}{2} t  + \frac{2 \tau^2}{t} \left( \frac{\rho_1^-}{\rho_2} +\frac{3D}{2\rho_2d} \ln(2)\right)\right)\right)\\
\leq &  \frac{C_5(\ve)}{\mu(B(x,\sqrt t))^\frac{1}{2} \mu(B(y,\sqrt t))^\frac{1}{2}} \exp(C_6(\ve) \rho_1^- t) \exp\left(-\frac{d^2(x,y)}{(4+\ve)t}\right).\\
\end{align*}

Therefore we have
\begin{align*}
 & A\left( \frac{\sqrt t}{2} \right)^\frac{D}{2}  2^{-\frac{D}{2}-2}  \exp\left( -\frac{d\rho_1^- t}{4}\right)
\exp\left(-\frac{d_\tau\left(x,y\right)^{2}}{2t}\left(\frac{D}{d} +\frac{\rho_1^-}{2} t  + \frac{2\tau^2}{t} \left( \frac{\rho_1^-}{\rho_2} +\frac{3D}{2\rho_2d} \ln(2)\right)\right)\right)\\
\leq &  C_5(\ve) \exp(C_6(\ve) \rho_1^- t) \exp\left(-\frac{d^2(x,y)}{(4+\ve)t}\right).\\
\end{align*}
Thus for all $t>0$: 
\begin{align*}
0 \leq & -\frac{D}{2} \ln A\left( \frac{\sqrt t}{2} \right) + \left(\frac{D}{2}+2\right) \ln 2 + \ln C_5(\ve) + C_6(\ve) (1+\rho_1^-) t + \frac{d\rho_1^- t}{4} \\
 & -\frac{d^2(x,y)}{(4+\ve)t} + \frac{d_\tau\left(x,y\right)^{2}}{2t}\left(\frac{D}{d} + \frac{\rho_1^-}{2} t   +\frac{2\tau^2}{t} \left( \frac{\rho_1^-}{\rho_2} +\frac{3D}{2\rho_2d} \ln(2)\right)\right).
\end{align*}

Since $-\ln A\left( \frac{\sqrt t}{2} \right) \leq \ln \left(1+ \frac{d \rho_1^-  t}{4}\right) -\ln C_1 \leq \frac{d \rho_1^-  t}{4} -\ln C_1$,  fixing $\ve=1$, 
there exist some constants $E_1, E_2$ which only depend on $d,\kappa$ and $\rho_2$ such that for all $t>0$,
we have for all $x,y\in \M, t>0$ and $\tau>0$ 
$$
0\leq E_1  +  E_2  \rho_1^- t - \frac{d^2(x,y)}{(4+\ve)t} + \frac{d_\tau\left(x,y\right)^{2}}{2t}\left(\frac{D}{d} +\frac{\rho_1^-}{2}t + \frac{2\tau^2}{t}\left( \frac{\rho_1^-}{\rho_2} +\frac{3D}{2\rho_2d} \ln(2)\right)\right).
$$
Therefore, for some positive constants $A_i, 1\leq i \leq 3$ which only depend on $d,\kappa$ and $\rho_2$,
\[
d(x,y)^2 \leq A_1 ( 1+ \rho_1^- \, t)\, t + A_2 ( 1+ \rho_1^- \, t)\, d_\tau(x,y)^2 +A_3 (1+ \rho_1^-)\,   \frac{\tau^2 d_\tau(x,y)^2}{t}. 
\]
Since $\tau\leq 1$, if $d_\tau(x,y) \leq 1$, choosing $t= \tau d_\tau(x,y) \leq 1$ yields
\[
d\left(x,y\right)^2\leq  (1+ \rho_1^-) \left((A_1+A_3)\tau d_\tau(x,y)  +A_2 d_{\tau}\left(x,y\right)^2\right) \leq  (A_1+A_2+A_3)  (1+ \rho_1^-) d_\tau (x,y).
\]
If  $d_\tau(x,y) \geq 1$, choosing $t=\tau \leq 1 $, we infer
$$
d(x,y)^2 \leq  (1+\rho_1^-) \left(A_1 \tau+ A_2  d_\tau(x,y)^2  + A_3 \tau d_\tau (x,y)^2\right) \leq (A_1+A_2+A_3)   (1+ \rho_1^-) d_\tau (x,y)^2.
$$ 
\end{proof}

\section{Subelliptic estimates}
In this section, we investigate some consequences of the curvature-dimension inequality. 
We show that if an  operator $L$ is an H\"ormander type operator and if it satisfies some curvature-dimension inequality CD$(\rho_{1},\rho_{2},\kappa,d)$, then it is necessarly a rank 2 operator.
The proof is based on the distance comparison theorem (Theorem \ref{th-distance}). Actually, only a local distance comparison is needed. 
The notions of H\"ormander type operator and rank 2 operator are explained below. 

\

First, the comparison  principle of Fefferman and Phong between sub-elliptic and elliptic  balls (see \cite{FP}) implies the following local sub-elliptic estimate: 
\begin{theorem}\label{th-sub-ell-est}
Assume the operator $L$ satisfies the condition CD$(\rho_{1},\rho_{2},\kappa,d)$
for some  $\rho_{1}\in\mathbb{R}, \rho_{2}>0, \kappa\geq 0,$ and $d\geq 2$. Assume moreover that the metric associated to $d_\tau$ is a Riemannian metric $g_\tau$ on $\M$ for some $\tau$. 
Let $\Omega$ be a bounded domain  in $\M$ and a  chart $\phi:U\subset \mathbb R^m \to \Omega$. Then there exist some  constants $c=c(\Omega, \phi)>0$ and $C=C(\Omega, \phi)>0$ such that
\begin{equation}\label{sub-ell-est}
\Vert Lu \Vert + C\Vert u \Vert \geq c \Vert u \Vert_{(1)},  u\in \mathcal C_0^\infty (\phi(U)),
\end{equation}
where $\Vert \cdot \Vert$ denotes the usual $L^2(\mu)$ norm in $\Omega$ and where $\Vert \cdot \Vert_{(s)}$ is the standard Sobolev norm.
\end{theorem}

\begin{proof} By Theorem \ref{th-distance}, there exists a constant $A$ such that $d(x,y)\leq A \sqrt {d_\tau(x,y)}$ for all $x,y \in \Omega$. Therefore,  $B_\tau(x,R) \subset B(x,A \sqrt R)$. 
Since $\bar \Omega$ is a compact set, the metric $g_\tau$ is comparable with $g_{eucl}$ the metric obtained from the Euclidean one in $U$ by the map $\phi$. 
If we pull back the result in $U$, we thus have, $B_{eucl}(x,R) \subset B(x,A' \sqrt R)$ for some constant $A'>0$.
The result then follows from Theorem 1 in \cite{FP}.
\end{proof}

We call an H\"ormander type operator an operator $L$ which satisfies the general assumptions of Section 2 and which can be written locally as  $L= \sum_{i=1}^r X_j^* X_j$ for some $\mathcal C^\infty $ vector fields $X_j$.
 
We say it is an operator of rank $k$ if the vector fields and their commutators up to order $k$: 
$$
X_1,\dots ,X_r,  [X_{i_1},X_{i_2}], \dots, [X_{i_1},[X_{i_2},[ \dots,X_{i_k}]] \dots ], i_j=1 \dots r 
$$  
 span the tangent space in each point of $\M$. 

\

The following theorem is a direct consequence of an important result of Rothschild and Stein \cite{RS} and of the local subelliptic estimate (\ref{sub-ell-est}) (see Theorem 2.1 in \cite{JSC} for more details). 

\begin{theorem}
In addition of the hypothesis of Theorem \ref{th-sub-ell-est}, assume that the operator L is an H\"ormander type operator. Then $L$ is a rank 2 operator. 
\end{theorem}

\section{Gromov's precompactness theorem}

The goal of this section is to establish a generalisation Gromov's precompactness theorem for our class of subriemannian manifolds.
Initially, the Gromov's precompactness theorem states that the space of Riemannian manifolds with Ricci curvature bounded below by $k$, dimension bounded  by $N$ and diameter less  $D$ is precompact for the Gromov-Hausdorff convergence. Moreover the result can be extended for the (pointed) measured Gromov-Hausdorff convergence by endowing the Riemannian manifolds with their Riemannian volume. We refer to the book of Villani \cite{V}, chapter 27 for the statement of the result and the careful definitions of the above convergences.     

In our generality, contrary to the Riemannian case,  the measure $\mu$ is only defined up to a positive constant. Here, we need to normalize the measure. 

Let $\M$ be a compact smooth manifold and $\mu$ be a smooth measure on $\M$ such that there exists a smooth second order sub-elliptic differential operator $L$ which satisfies the general assumptions described in section 2. Let  us assume that the measure satisfies the normalisation property $\mu(\M)=1$. We say the compact  metric measured space $\M=(\M,\mu)$ belongs to  $ \mathcal M_R (\rho_{1},\rho_{2},\kappa,d) $, $R>0$ if moreover $L$ satisfies CD$(\rho_{1},\rho_{2},\kappa,d)$ and the (sub-Riemannian) diameter of $\M$ is bounded above by $R$.

\begin{theorem}
 Let $\rho_{1}\in\mathbb{R}, \rho_{2}>0, \kappa\geq 0,$ and $d\geq 2$, $R>0$.
The set of metric measured spaces $\mathcal M_R (\rho_{1},\rho_{2},\kappa,d)$ is precompact for the  measured Gromov-Hausdorff convergence.
\end{theorem}

\begin{proof}
The proof is an easy consequence of Theorem 27.31 in \cite{V} and of the doubling property of Theorem \ref{th-doubling}. 
\end{proof}

%\red{I do not see what to change.}

\end{document}